\documentclass[a4paper,11pt]{amsart}
\usepackage{amssymb,amsthm,bbm}

\DeclareMathOperator{\GL}{\mathrm{GL}}
\DeclareMathOperator{\Irr}{\mathrm{Irr}}
\DeclareMathOperator{\Ker}{\mathrm{Ker}}
\DeclareMathOperator{\Hom}{\mathrm{Hom}}

\DeclareMathOperator{\tr}{\mathrm{tr}}
\DeclareMathOperator{\dd}{\mathrm{d}}
\DeclareMathOperator{\End}{\mathrm{End}}

\DeclareMathOperator{\gr}{\mathrm{gr}}

\def\ii{\mathrm{i}}
\def\dd{\mathrm{d}}
\def\Id{\mathrm{Id}}
\def\C{\mathbbm{C}}

\def\Q{\mathbbm{Q}}
\def\Z{\mathbbm{Z}}
\def\la{\lambda}
\def\om{\omega}

\newtheorem{prop}{Proposition}[section]
\newtheorem{lemma}[prop]{Lemma}

\newtheorem{theo}[prop]{Theorem}
\newtheorem{quest}[prop]{Question}
\DeclareMathOperator{\Vect}{\mathrm{Vect}}

\def\un{\mathbbm{1}}
\newcommand{\into}{\hookrightarrow}
\newcommand{\onto}{\twoheadrightarrow}

\title{Knizhnik-Zamolodchikov bundles are topologically trivial}
\author{Ivan Marin}
\date{September 14, 2008}
\begin{document}

\maketitle

\bigskip
\begin{center}
Institut de Math\'ematiques de Jussieu \\
Universit\'e Paris 7 \\
175 rue du Chevaleret \\
F-75013 Paris
\end{center}
\bigskip

\noindent {\bf Abstract.} We prove that the vector bundles at the core of the
Knizhnik-Zamolodchikov and quantum constructions of braid groups
representations are topologically trivial bundles. We provide
partial generalizations of this result to generalized braid groups.
A crucial intermediate result is that the representation ring of the
symmetric group on $n$ letters is generated by the alternating powers of its natural
$n$-dimensional representation.
\bigskip

\noindent {\bf Keywords.}  Knizhnik-Zamolodchikov, braid groups, reflection groups, K-theory.

\bigskip

\noindent {\bf MSC 2000.} 57R22, 20F36, 20C30.

\section{Introduction}

The braid group $\mathcal{B}_n$ on $n$ strands can be realized
as the fundamental group of  $\C_*^n/\mathfrak{S}_n$, where
$\C_*^n = \{ (z_1,\dots,z_n) \in \C^n \ | \ i \neq j \Rightarrow
z_i \neq z_j \}$ and $\mathfrak{S}_n$ is the symmetric group
on $n$ letters, acting on $\C^n$ by permutation of coordinates.
Representations of the braid groups appear in a number of ways,
notably in the study of quantum groups and low-dimensional topology.

A well-known way to obtain and study these representations is
to consider them, when possible, as the monodromy of a flat
connection on (complex) vector bundles of the form $(\C_*^n \times U \to \C_*^n) /\mathfrak{S}_n$,
where $U$ is a finite-dimensional vector space equipped with a linear
action of $\mathfrak{S}_n$, $\C_*^n \times U \to \C_*^n$
denotes the corresponding trivial vector bundle over $\C_*^n$,
and $\mathfrak{S}_n$ acts on $\C_*^n \times U$ by diagonal
action. Examples include, but are not restricted to,
the case when $U = A^{\otimes n}$ is
the $n$-times tensor power of some Lie algebra representation
$A$ and $\mathfrak{S}_n$ acts on $A^{\otimes n}$ by permuting the
tensor factors.

A natural question is whether the \emph{quotient} vector bundles themselves are
or can be topologically trivial. We call such quotient bundles
\emph{Knizhnik-Zamolodchikov bundles}, or KZ-bundles for short. The primary purpose of this work is to answer (positively) to this question. 

\begin{theo} \label{theoSN} Every Knizhnik-Zamolodchikov bundle is topologically
trivial.
\end{theo}

A consequence of the Kohno-Drinfeld theorem is thus that
vector bundles over $\C_*^n/\mathfrak{S}_n$ associated to
the braid groups representation afforded by $R$-matrices
are also topologically trivial.

\medskip

The proof of this theorem goes as follows. 
First note that we may replace $\C_*^n \simeq \C \times V$
by $V$ where $V = \{ (z_1,\dots,z_n ) \in \C_*^n \ | \ z_1 + \dots + z_n= 0 \}$
and the 1-dimensional subspace denoted $\C$ is spanned by the vector $(1,\dots,1)$.
Step (1) is to explicitely
trivialize the bundle when $U = V$. This
trivialization had already been observed in \cite{ACC}, where
the authors used the `Vandermonde spiral' introduced by Arnold in \cite{ARNOLD}.
We give a proof that generalizes this fact to arbitrary reflection groups
(see below), and provides an explanation of the appearance of the
Vandermonde matrix, as a jacobian determinant originating from
invariants theory. 

Step (2) is the elementary remark that irreducible representations of dimension at least 2 of the symmetric
groups $\mathfrak{S}_n$ have actually dimension 
large enough so that the corresponding vector bundles are in the `stable
range', meaning that they are trivial if and only if they are
stably trivial. The problem can thus be reduced to the computation
of the image of the representation ring $R(\mathfrak{S}_n)$
inside the reduced (complex) K-theory of $\C_*^n$.

Step (3) is that $R(\mathfrak{S}_n)$ is generated as a ring by the
representations $\Lambda^k V$ (theorem \ref{theoGenSn} below). As far as we know, this is a new result
that might be applied elsewhere. Since the alternating powers
of a trivial bundle are trivial, this concludes the proof of
the theorem.

\medskip

This problem can be placed in the following larger setting.
Let $V$ be a finite-dimensional complex vector
space and let $W < \GL(V)$ be a finite group generated by (pseudo-)reflections,
i.e. finite-order endomorphisms fixing some hyperplane. Such a group
is called a reflection group.
For convenience, we assume that $W$ acts irreducibly on $V$.
To such a reflection group are associated the arrangement $\mathcal{A}$ of the
hyperplanes fixed by the reflections in $W$, an hyperplane complement
$X = V \setminus \bigcup \mathcal{A}$, (generalized) pure braid group
$P =\pi_1(X)$ and braid group $B = \pi_1(X/W)$. Generalized
KZ-systems are then flat connections on the quotient bundles $(X \times U
\to X)/W$ where $U$ is a finite-dimensional representation of $W$.
This yields the following question.

\begin{quest} Are all generalized KZ-bundles topologically trivial ?
\end{quest}

Following the same demarch as in the proof of the theorem, we prove that
steps (1) and (2) are valid in this more general setting (see propositions \ref{proptrivref} and \ref{propdimn2}). Complications
arise at step (3), because it is no more true in general that the representation
ring $R(W)$ is generated by the exterior powers of the reflection
representations. However, using several ad-hoc means, we manage to
extend the result to several irreducible reflection groups, following their
Shephard-Todd classification (see \cite{ST}) :

\begin{theo} \label{theoW} Let $W$ denote an irreducible reflection group.
Generalized KZ-bundles are topologically trivial whenever
\begin{itemize}
\item $W$ has rank 2
\item $W$ is one of the exceptional groups $G_{23}=H_3$, $G_{24}$, $G_{25}$,
$G_{26}$, $G_{30}=H_4$, $G_{32}$, $G_{33}$,
$G_{35}=E_6$.
\end{itemize} \end{theo}

\medskip

Back to the ordinary braid groups, one interest of this last theorem
is that it applies to the exceptional Shephard groups
$G_4,G_8,G_{16},G_{25},G_{32}$, whose braid groups are ordinary
braid groups (see e.g. \cite{BMR}). It follows that the corresponding representations
of $\mathcal{B}_3,\mathcal{B}_4,\mathcal{B}_5$ also have trivial
image in the reduced K-theory of the group, and more precisely
that their associated vector bundles over $\C_*^n/\mathfrak{S}_n$ are
topologically trivial.

\medskip

\noindent {\bf Acknowledgements.} I thank Bruno Klingler and
Julien March\'e for several useful discussions on homotopy theory
and vector bundles. I also thank Rapha\"el Rouquier, Pierre
Vogel and Christian Blanchet for interesting comments or advices.

\section{Topological preliminaries}

Let $Y$ denote a connected paracompact topological space which has
the homotopy type of a $n$-dimensional CW-complex, $\tilde{Y}$
its universal cover, $B = \pi_1(Y)$. Let $X$ be a finite Galois covering
of $Y$ with transformations group $W$, so that $Y = X/W$ and
$X = \tilde{Y}/P$, where $P= \Ker(B \onto W)$.

All vector bundles considered in this paper are complex ones.
We let $\Vect_Y$ denote the category of vector bundles over $Y$, and
$\Vect(Y)$ denote the semiring of (isomorphism classes of) vector bundles
over $Y$. We let $K(Y)$ denote the corresponding K-theoretic ring, and $\tilde{K}(Y)$
the quotient of $K(Y)$ by its subring of trivial bundles. We
refer to \cite{HUSE} for these classical notions and their basic
properties.

Each linear representation $\rho : B \to \GL_N(\C)$ defines a vector
bundle $\mathcal{V}(\rho) = Y \times_{\rho} \C^N \ to Y$, defined by
$Y \times_{\rho} \C^N = (\tilde{Y} \times \C^N)/B$, where $B$ acts
on $\tilde{Y} \times \C^N$ by $g.(y,z) = (g.y,\rho(g).z)$. We let
$R^+(B)$ and $R(B)$ denote the representation semiring and representation
ring of $B$, respectively. The correspondance $\rho \mapsto \mathcal{V}(\rho)$
defines a semiring morphism $\Phi^+ : R^+(B) \to \Vect(Y)$ and
a ring morphism $\Phi : R(B) \to K(Y)$.

The group morphism $B \onto W$ induces a ring morphism $R(W) \to R(B)$
and we still denote $\Phi : R(W) \to K(Y)$ the composite
morphism. Let $\rho : W \to \GL_N(\C)$ be a representation of $W$. Then
$\Phi(\rho)$ is the class of $\mathcal{V}(\rho)$, and
$$
Y \times_{\rho} \C^N = (\tilde{Y} \times \C^N)/B = ((\tilde{Y} \times \C^N)/P)/W
= (X \times \C^N)/W,
$$
as $P$ acts trivially on $\C^N$, hence $\mathcal{V}(\rho)$
is the quotient of the trivial vector bundle $X \times \C^N \to X$
by $W$ acting on $X \times \C^N$ through $g.(x,z) = (g.x,\rho(g).z)$.

We state as a proposition a recollection of classical results to be
used in the sequel. Recall that to each partition $\la$ of some
integer $r$ are associated Schur endofunctors $\mathcal{S}_{\la}$ of the category of
finite-dimensional vector spaces. Being continuous (in the
sense of \cite{HUSE} \S 5.6) they extend to endofunctors
$\mathcal{S}_{\la}$ of $\Vect_Y$. Our convention is that $\mathcal{S}_{[1^r]}
(\xi) = \Lambda^r \xi$. We moreover assume that $B$ is finitely generated,
so that the topology of $\Hom(B,\GL_N(\C))$ is naturally defined.

\begin{prop} \label{propbase} Let $\rho$ be a (finite-dimensional) linear representation of $W$.
\begin{enumerate}
\item If $N = \dim \rho \geq n/2$, then $\Phi(\rho) \equiv 0$ in $\tilde{K}(Y)$
if and only if $\mathcal{V}(\rho)$ is a trivial vector bundle.
\item If $t \mapsto \rho_t$ is a continuous path in $\Hom(B,\GL_N(\C))$ then
$\Phi^+(\rho_0) = \Phi^+(\rho_1)$.
\item If $\mathcal{V}(\rho)$ is a trivial bundle, then so is
$\mathcal{S}_{\la}(\mathcal{V}(\rho)) \simeq \mathcal{V}(\mathcal{S}_{\la} \rho)$.
\end{enumerate}
\end{prop}

\begin{proof}
The assumption in (1) means that $\Phi^+(\rho) + \un_Y^r = \un_Y^s$
for some $r \geq 0$ and $s = r + \dim \rho$, where $\un_Y$ denotes
the trivial line bundle over $Y$. Then conclusion that $\Phi^+(\rho)
= \un_Y^{\dim \rho}$ follows from \cite{HUSE} \S 9, theorem 1.5.
For fact (2), see \cite{HUSE} \S 3.4. The proof of fact (3) is
similar to the more classical special case $\Lambda^r \mathcal{V}(\rho)
\simeq \mathcal{V}(\Lambda^r \rho)$ dealt with in \cite{HUSE} \S 13.
\end{proof}

\section{Representations of complex reflection groups}

In this section we establish general results for arbitrary reflection
groups. Let $V$ be a $n$-dimensional vector space, $W < \GL(V)$ be a finite (pseudo-)reflection
group, with associated hyperplane arrangement $\mathcal{A}$.
We denote $X = V \setminus \bigcup \mathcal{A}$, $P = \pi(X)$,
$B = \pi(X/W)$ the pure braid group and braid group associated to $W$. To each $H \in \mathcal{A}$ we associate the
order $e_H$ of the (cyclic) fixer of $H$ in $W$, and call
distinguished reflection associated to $H$
the pseudo-reflection with set of fixed points $H$ and
eigenvalue $\exp(2 \ii \pi/e_H)$.

The spaces $X$ and $X/W$ are $K(\pi,1)$ (see \cite{BESSIS}). Moreover,
$X/W$ is the complement of a finite number of hypersurfaces, hence
is an affine variety of complex dimension $n$. By \cite{MILNOR} theorem 7.2
it follows that the connected space $Y = X/W$ has the homotopy type
of a $n$-dimensional CW-complex, and the results of the above section
can be applied.

\subsection{One-dimensional representations}

There is a distinguish class of
preimages in $B$ of distinguished pseudo-reflection that are called braided
reflections (see e.g. \cite{ARRREFL}). We denote $\om_H = \dd \alpha_H / \alpha_H$
for an arbitrary linear form $\alpha_H$ with kernel $H$. If $s$
is the distinguished reflection associated to $H$ we may note
$\om_s   = \om_H$.

\begin{prop} \label{propdim1}  If $\rho_0 \in \Irr(W)$ has dimension 1, then
$\mathcal{V}(\rho_0)$ is a trivial bundle.
\end{prop}
\begin{proof}
We consider the bundle $X \times \C \to X$, and the family of 1-form
$\om_h = h_s \sum_s \rho_0(s) \om_s \in \Omega^1(X) \otimes \End(\C)$,
with $h_s \in \C$ and $h_s = h_s'$ whenever $s,s'$ are conjugate in $W$.
This 1-form is thus indexed by a tuple $\underline{h}$ in $\C^r$
where $r$ is the number of conjugacy classes of distinguished reflections.
It is easily checked to be integrable and $W$-equivariant w.r.t.
the diagonal action of $W$ on $X \times \C$, where $W$ acts on $\C$ through
$\rho_0$. It thus defines flat connections on the quotient bundle
$(X \times \C \to X)/W$, for which the monodromy is given by $\rho_{\underline{h}} : B \to
\GL(\C)$, with $\rho_{\underline{h}}(\sigma) = \rho_0(s) \exp( 2 \ii \pi h_s \rho_0(s)))$,
if $\sigma$ is a braided reflection with $\pi(\sigma) = s$.
There obviously exists a collection of real scalars $h_s$ such that $\rho_{\underline{h}}(\sigma) = \Id$ for all such
distinguished reflections. Since $B$ is generated by the corresponding
braided reflections we get a continuous map $\C^r \to \Hom(B,\GL(\C))$
connecting the trivial representation and $\rho_0$. It follows
from proposition \ref{propbase} (2) that $\mathcal{V}(\rho_0)$ is a trivial bundle.
\end{proof}

Using more information on the structure of $B$, it is actually possible
to show that any 1-dimensional representation of $B$ corresponds to a
trivial bundle on $X/W$. Indeed, the abelianization $B/(B,B)$
is isomorphic to $\Z^r$ where $r$ is the number of classes of hyperplanes
in $\mathcal{A}$ (see \cite{BMR}). It follows that
$\Hom(B, \GL(\C)) \simeq (\C^{\times})^r$ is connected, hence any 1-dimensional
representation can be deformed to the trivial one.

\subsection{A remark on holomorphic line bundles}

We emphasize the fact that there is no hope to generalize the results
exposed here to arbitrary (non necessarily flat) vector bundles, by noticing
the following facts on 1-dimensional holomorphic vector bundles on $X/W$.
First recall that $X/W$ is a complement of hypersurfaces, hence is affine
and in particular is a Stein manifold. Letting $\mathcal{O}$ denote the sheaf
of holomorphic functions on $X/W$, the exponential
exact sequence $0 \to  \Z \to \mathcal{O} \to \mathcal{O}^{\times} \to 0$
and Cartan theorem B yields an isomorphism $H^1(X/W,\mathcal{O}^{\times})
\to H^2(X/W, \Z)$ given by the first Chern class. On the other
hand, $H^1(X/W,\mathcal{O}^{\times})$ classifies rank 1 holomorphic bundles
on $X/W$. In case $B$ is an ordinary braid group, we have $H^2(B,\Z) = 0$
hence all 1-dimensional bundles are holomorphically trivial.
This fact however does not extend to arbitrary reflection groups,
and the nonvanishing of $H^2(X/W,\Z)$ in general will illustrate
the specificity of the results exposed here.

We first recall the following facts. First of all, $X/W$ is a $K(\pi,1)$,
hence $H^2(X/W,\Z) = H^2(B,\Z)$. Moreover, any such braid group
$B$ is obtained as $\pi_1(X/W)$ for some reflection group
$W$ with reflections of order 2. Finally, $H_1(B,\Z) = B/(B,B)$ is torsion-free,
hence the universal coefficients exact sequence $0 \to \mathrm{Ext}(H_1(B),\Z)
\to H^2(B,\Z) \to \Hom(H_2(B),\Z) \to 0$ implies that $H^2(B,\Z)$
is a free $\Z$-module, of rank $b_2 = \dim H^2(B,\Q)$. It follows that
we only need to compute the second Betti number $b_2$ when $W$ is a 2-reflection
group to get $H^2(X/W,\Z)$ in general.

The Shephard-Todd classification (see \cite{ST}) divides the irreducible
reflection groups into the infinite series $G(de,e,n)$ and 34 exceptional groups $G_4,\dots,G_{37}$.
For $d,e,r \geq 1$, the group $G(de,e,r)$
is defined as the group of $r \times r$ monomial matrices
whose nonzero entries lie in $\mu_{de}(\C)$ and have
product in $\mu_d(\C)$. Inside this infinite series, the groups whose reflections have order 2
are the groups $W = G(e,e,n)$ and $W = G(2e,e,n)$.

\begin{prop} Let $W$ be an irreducible 2-reflection group with associated
braid group $B$. Then $H^2(B,\Z) = \Z^{b_2}$ with
\begin{itemize}
\item $b_2 = 2$ if $W$ has type $G_{28} = F_4$, $G(2e,e,n)$ with
$n \geq 3$, or $G(2e,e,2)$ with $e$ even.
\item $b_2 = 1$ if $W$ has type $G_{13},G_{23},G_{24},G_{27}$,
$G(e,e,2)$ with $e$ even, or $G(2e,e,2)$ with $e$ odd.
\item $b_2 = 0$ otherwise.
\end{itemize}
\end{prop}
\begin{proof}
For the infinite series $G(e,e,n)$ and $G(2e,e,n)$, we deduce
these Betti numbers from \cite{LEHRER}, where they are implicitely
computed. More precisely, corollary 6.5 in \cite{LEHRER} states
that the Poincar\'e polynomial of $X/W$ is $1+t$ if $e$ or $n$ is
odd, and $1+t+t^{n-1}+t^n$ is $e$ and $n$ are even ; for the
series $G(2e,e,n)$, letting $\tilde{Y} = \{ (z_1,\dots,z_n) \ | \ z_i \neq 0,
z_i/z_j \not\in \mu_{2e}(\C) \} / G(2e,2e,n)$, theorem 6.1 in \cite{LEHRER} decomposes
the rational cohomology groups of $\tilde{Y}$ under the action of $G(2e,1,n)/G(2e,2e,n)
\simeq \mu_{2e}(\C)$. In this case $X/W = \tilde{Y} / G(2e,e,n)$ hence
$H^*(X/W, \Q) = H^{*}(\tilde{Y},\Q)^{\mu_2(\C)}$ where $\mu_2(\C)$ is identified
with $G(2e,e,n) /G(2e,2e,n)$. The action of $G(2e,1,n)$ on $H^*(\tilde{Y},\Q)$
factorizes through the projection $\delta : G(2e,1,n) / G(2e,2e,n) = \mu_{2e}(\C) \to \mu_2(\C)$.
We have $\delta(\mu_2(\C)) = \{ 1 \}$ if $e$ is even, and $\delta(\mu_2(\C)) = \mu_2(\C)$ is
$e$ is odd. A consequence of this theorem 6.1 is thus that
the Poincar\'e polynomial of $X/W$ is
$$
 \left\lbrace \begin{array}{ll}
(1+t)(1+t+t^2+ \dots + t^{n-1}) + t^{n-1} + t^n & \mbox{ if $n$ and $e$ are even} \\
(1+t)(1+t+t^2+ \dots + t^{n-1}) & \mbox{otherwise.}
\end{array} \right.
$$
The conclusion follows for the infinite series. For
the exceptional groups, using $H^*(X/W,\Q) = H^*(X,\Q)^W$
and the combinatorial description of $H^*(X,\Q)$,
it is shown in \cite{ORLIKTERAO} (corollary 6.17) that
$
b_2 = \sum_{Z \in T_2} |\mathcal{H}_Z / W_Z | - 1
$
where $T_2$ is a system of representatives for the codimension
2 subspaces in the hyperplane arrangement lattice under the
action of $W$ and, for $Z \in T_2$, $\mathcal{H}_Z = \{ H \in \mathcal{A} \ | \ H \supset Z \}$
and $W_Z = \{ w \in W \ | \ w(Z) = Z \}$. Direct computations for all exceptional
2-reflection groups concludes the proof.
\end{proof}

\subsection{Reflection representations}
We consider the diagonal
action of $W$ on $X \times V$ and the trivial bundle on $X \times V \to X$
with the corresponding action.

\begin{prop} \label{proptrivref}
The quotient bundle $(X \times V \to X)/W$ is holomorphically
equivalent to the trivial bundle $(X/W)\times V \to (X/W)$.
\end{prop}
\begin{proof}
Identifying $V$ with $\C^n$, the algebra $\C[V] = S(V^*)$
of polynomial functions on $V$ is acted upon by $W$, and
the Chevalley-Shephard-Todd theorem (see e.g. \cite{BENSON}) states that
$\C[V]^W$ is the algebra of polynomial in $n$
algebraically independant polynomials $f_1,\dots, f_n$. We then
define $f : V \to V$ by $f =(f_1,\dots,f_n)$ and consider
$\Psi : V \times V \to V \times V$ defined by
$\Psi(x,v) = (x,\dd_x f(v))$, where $\dd_x f \in \End(V)$
is the differential of $f$ at $x$. We have $f(w.x) = x$
for all $x \in V, w \in W$, hence $\dd_{w.x}f (v) = 
\dd_x f(w^{-1}.v)$. It follows that
$\Psi(w.(x,v)) = (w.x, \dd_{w.x} f(w.v)) = (w.x , \dd_x f(v))$.
On the other hand, $\dd_x f \in \GL(V)$ for all $x \in X$ if and only if
the zero locus of $x \mapsto \det( \dd_x f) = \det( \frac{\partial f_i}{\partial x_j} )$
is contained in $\bigcup \mathcal{A}$. This is well-known
to be the case (see e.g. \cite{BENSON} \S 7.3), hence $\Psi$
defines an holomorphic automorphism of the trivial bundle $X \times V \to V$
which induces after quotient by $W$ an holomorphic isomorphism
between $(X \times V \to X)/W$ and $(X/W) \times V \to (X/W)$.
\end{proof}

If $W$ is the symmetric group $\mathfrak{S}_{n+1}$ acting on $\C^{n+1} = V \times \C$
with $V = \{ (z_1,\dots,z_{n+1}) \in \C^{n+1} \ | \ z_1 + \dots + z_{n+1} = 0 \}$
then the $f_1,\dots,f_n$ can be chosen to be the power sums
$f_k = x_1^k+\dots+x_{n+1}^k$. In that case, up to a rescaling,
the map $x \mapsto \dd_x f$ is the Vandermonde spiral of Arnold,
and yields the trivialization of \cite{ACC}.

\subsection{A lower bound for the dimensions of the representations}

We recall that, for $d,e,r \geq 1$, the group $G(de,e,r)$
is defined as the group of $r \times r$ monomial matrices
whose nonzero entries lie in $\mu_{de}(\C)$ and have
product in $\mu_d(\C)$. In particular
$G(de,e,r)$ contains the symmetric group $\mathfrak{S}_r$
of permutation matrices, and we have $G(de,e,r)
= \mathfrak{S}_r \ltimes D_r$, where $D_r$ is the
subgroup of diagonal matrices in $G(de,e,r)$.
We embed $G(de,e,r-1)$ in $G(de,e,r)$ by leaving fixed
the last basis vector.

\begin{lemma} Let $d,e \geq 1$ and $W_r = G(de,e,r)$ with $r \geq 4$.
If $p: W_r \to Q$ is a group morphism such that $p(W_{r-1})$ is commutative
then $p(W_r)$ is commutative.
\end{lemma}
\begin{proof}
The kernel of $p$ contains the derived subgroup
$(\mathfrak{S}_{r-1},\mathfrak{S}_{r-1})$, which contains a 3-cycle
since $r \geq 4$. Since $\Ker p$ is normal it thus contains
the alternating group $\mathfrak{A}_r$. In particular
$p(\mathfrak{S}_r)$ is commutative, its action on $p(D_r)$
factors through the sign character, and is trivial on $p(D_{r-1})$.
Let $\zeta$ be a primitive $de$-th root of 1 and
$x = \mathrm{diag}(\zeta^{-1},1,\dots,1,\zeta)$. It is clear that
$D_r$ is generated by $x$ and $D_{r-1}$. But
$x$ is conjugate to some element of $D_{r-1}$ by some 3-cycle,
hence $p(x) \in p(D_{r-1})$ and $p(D_r) =p(D_{r-1})$. The conclusion
follows.
\end{proof}

\begin{lemma} If $\rho$ is an irreducible representation
of $G(de,e,r+1)$ with $\dim \rho > 1$, then its restriction
to $G(de,e,r-1)$ is not irreducible.
\end{lemma}

\begin{proof}
By contradiction we assume otherwise. Then
its restriction to $G(de,e,r)$ is already irreducible.
By \cite{BRANCHING} propositions 3.1 and 3.2,
in that case $\rho$ can be extended to a representation of
$G(de,1,r)$, so we can assume $e=1$. Since $G(d,1,r+1)$
is a wreath product, $\rho$ is classically associated to
a multipartition ($d$-tuple of partitions). By the branching rule
for such wreath products, since the restriction to $G(d,1,r)$
is irreducible, the multipartition has only one part, which has the
form $[a^b]$ (that is to say, its Young diagram is a rectangle),
with $ab = r+1$. Since $\dim \rho > 1$ we have $a,b \geq 2$. But then
its restriction to $G(d,1,r-1)$ is not irreducible,
because it contains the components labelled by $[a^{b-1},a-2]$ and
$[a^{b-2},a-1,a-1]$, which are distinct. This contradiction concludes
the proof.
\end{proof}

\begin{prop} \label{propdimn2} Let $W$ be a finite irreducible pseudo-reflection
group of rank $n$. Then every irreducible complex representation
$\rho$ of $W$ with $\dim \rho > 1$ has dimension at least $n/2$.
\end{prop}
\begin{proof}
We use the Shephard-Todd classification and a case-by-case examination
to verify that the statement holds for the 34 exceptional
groups. We then assume $W = W_r = G(de,e,r)$, and we show
that $\dim \rho > 1 \Rightarrow \dim \rho \geq r/2$
by induction on $r$. Since the statement is void for $r \leq 4$
we assume $r \geq 4$, and let $\rho \in \Irr(W_{r+1})$ with $\dim \rho> 1$.
This implies that $\rho(W_{r+1})$ is not commutative. By the lemma above
the restriction of $\rho$ to $W_{r-1}$ is not irreducible. If one of its
component has dimension at least 2, then by the induction hypothesis
$\dim \rho \geq 1 + (r-1)/2 = (r+1)/2$. On the other hand, if
all the components have dimension 1, then for $r \geq 4$ this would
imply that $\rho(W_r)$ and $\rho(W_{r+1})$ are commutative,
contradicting $\dim \rho > 1$. This concludes the proof of
the proposition.
\end{proof}

\section{The representation ring of the symmetric groups}

We denote $\mathfrak{S}_n$ the symmetric group on
$n$ letters, naturally embedded on $\mathfrak{S}_{n+1}$
by action on the $n$ first letters. We denote
$R_n = R_(\mathfrak{S}_n)$.
The basis in $R_n$ of irreducible representations is naturally indexed
by partitions $\la$ of $n$, represented by Young diagrams.
We let $V_{\la}$ denote the associated element of $R_n$.

We let $R_{\infty} = \bigoplus_{n=1}^{\infty} R_n$, and extend the
usual tensor product $\otimes$ on $R_n$ to $R_{\infty}$ by pointwise multiplication.
There is another well-known ring structure on $R_{\infty}$, given
by $V_{\la} \odot V_{\mu} = \sum_{\nu} L_{\la \mu \nu} V_{\nu}$,
where $L_{\la \mu \nu}$ are the well-known Littlewood-Richardon coefficients,
and $\nu$ is a partition of arbitrary size. By contrast,
the structure constants $V_{\la} \otimes V_{\mu} = \sum_{\nu}
C_{\la \mu \nu} V_{\nu}$, whose study has been initiated by
Murnaghan (1938), are notoriously complicated to understand.

A recent and remarkable discovery of Y. Dvir is that these two
products are related in the following way. For a partition $\la = [\la_1,\la_2,\dots]$ with $\la_i \geq \la_{i+1}$,
of $n$, define the partition $\theta(\la) = [\la_2,\la_3,\dots]$
of $n - \la_1$,
and let $d(\la) = |\theta(\la)| = \la_2 + \la_3 + \dots$.
It is a classical and elementary fact that $d(\la) = \min \{ r \geq 0 \ | \ V_{\la} \into
V^{\otimes r} \}$, where $V = V_{[n-1,1]}$ is the reflection
representation of $\mathfrak{S}_n$.
In particular $C_{\la,\mu,\nu} = 0$ whenever $d(\nu) > d(\la) + d(\mu)$.
Dvir's formula can be stated as follows

\begin{theo} (Dvir \cite{DVIR}, theorem 3.3) Let $\la,\mu,\nu$ partitions of $n$ such that
$d(\la) + d(\mu) = d(\nu)$.
Then $C_{\la,\mu,\nu} = L_{\theta(\la),\theta(\mu),\theta(\nu)}$.
\end{theo}

A way to interpret this formula is to introduce a filtration
$\mathcal{F}_n R_{\infty}$ of the
$\Z$-module $R_{\infty}$ defined by where $\mathcal{F}_n R_{\infty}$
is spanned by all $V_{\la}$ with $d(\la) \leq n$. Then 
$(\mathcal{F}_n R_{\infty})(  \mathcal{F}_m R_{\infty})
\subset \mathcal{F}_{n+m} R_{\infty}$. Considering the
induced product on $\gr R_{\infty}$,
Dvir formula says that this induced product is basically
given by the Littlewood-Richardson rule. 

\subsection{A generating set for $R(\mathfrak{S}_n)$}

The ring $(R_{\infty},\odot)$ is generated by the elements
$V_{[r]}$ for $r\geq 0$, which essentially means that
the symmetric polynomials are generated by the elementary ones.
Dvir formula allows us to derive the following analogous
fact.

\begin{theo} \label{theoGenSn} The ring $R(\mathfrak{S}_n)$
is generated by the $\Lambda^k V$ for $0 \leq k \leq n-1$.
\end{theo}
\begin{proof}
Recall that
$\Lambda^k V = V_{[n-k,1^k]}$, and notice that $\theta([n-k,1^k]) = [1^k]$.
Let $Q$ denote the subring of $R(\mathfrak{S}_n)$
generated by the $\Lambda^k V$. We prove that $V_{\la} \in
Q$ for all partition $\la$ of $n$ ($\la \vdash n$),
by induction on $d(\la) = |\theta(\la)|$. We have $d(\la) = 0
\Rightarrow \la = [n]  \Rightarrow V_{\la} = \Lambda^0 V$
and $d(\la) = 0
\Rightarrow \la = [n-1,1]  \Rightarrow V_{\la} = \Lambda^1 V$,
hence $V_{\la} \in Q$ if $d(\la) \leq 1$. We thus assume
$d(\la) \geq 2$.

We need some more notation. For a partition $\alpha$ of $m$,
we define the partition $\alpha^{\circ}$
by $\alpha^{\circ}_i = \max(0,\alpha_i - 1)$.
Clearly, $|\alpha^{\circ}| = |\alpha| - \alpha'_1$ 
hence $|\alpha^{\circ}|\leq |\alpha|$, with equality only
if $\alpha = 0$.

We let $\alpha = \theta(\la)$ and use another induction on $d(\la) - \alpha'_1 = |\alpha| - \alpha'_1$.
The case $d(\la) - \alpha'_1 = 0$ means $V_{\la} = \Lambda^{\alpha'_1} V
\in Q$, so we can assume $d(\la) - \alpha'_1 \geq 1$.

We let $r = |\alpha| - |\alpha^{\circ}|$.
Since $d(\la) \geq 2$ we have $\theta(\la) \neq 0$ and $r \geq 1$.
Moreover $\la_1 = n- |\alpha| \geq 0$, hence $n - |\alpha^{\circ}| \geq 0$.
We thus can introduce $\mu = [n-|\alpha^{\circ}|,\alpha^{\circ}_1,
\dots,]$ and
consider $M = V_{\mu} \otimes \Lambda^r V \in R(\mathfrak{S}_n)$.
Since $|\alpha^{\circ}| < |\alpha|$ we have $d(\mu) < d(\la)$
hence $V_{\mu} \in Q$ and $M \in Q$. Let $\nu \vdash n$
such that $V_{\nu} \into M$. If $d(\nu) < d(\mu) + r = d(\la)$
then $V_{\nu} \in Q$ by the first induction hypothesis.
Otherwise, $\nu_1 = n-|\alpha| = \la_1$, and
$C_{\mu,[n-r,1-r],\nu} = L_{\alpha^{\circ},[1^r],\theta(\nu)}$
by Dvir formula.
By the Littlewood-Richardon rule, we have
$L_{\alpha^{\circ},[1^r],\alpha} = 1$ and,
if $L_{\alpha^{\circ},[1^r],\theta(\nu)}$
is nonzero, then either $\theta(\nu)'_1 > \alpha'_1$,
in which case we know that $V_{\nu} \in Q$ by the second
induction hypothesis (as $\nu'_1 = 1+\theta(\nu)'_1$), 
or $\theta(\nu) = \alpha$. Hence $M \equiv V_{\la}$ modulo $Q$,
$V_{\la} \in Q$ and the conclusion follows by induction.
\end{proof}

\subsection{Proof of theorem \ref{theoSN}}

The above results are sufficient to give a proof of theorem \ref{theoSN}, as
sketched in the introduction. If $\rho$ is a representation
of $\mathfrak{S}_n$, it can be split as a direct sum $\rho_1 \oplus \dots \oplus
\rho_r$ of irreducible representations. As $\mathcal{V}(\rho)$ is the
Whitney sum of the $\mathcal{V}(\rho_i)$, we can assume that $\rho$
is irreducible. 

If $\rho$ has dimension 1, the reflection representation $V$, or
some alternating power of $V$,
then $\mathcal{V}(\rho)$ is trivial by propositions \ref{propdim1},
\ref{proptrivref} and \ref{propbase} (3).
In general, $\Phi(\rho)$ is a polynomial in the
$\Phi(\rho_0)$, for $\rho_0$ in the previous list,
by theorem \ref{theoGenSn}. It follows that $\Phi(\rho)$
belongs to the subring of trivial bundles. As $\C_*^n /\mathfrak{S}_n$
has the homotopy type of a $(n-1)$-dimensional
CW-complex, if $\dim \rho_0 \geq (n-1)/2$, this implies that
$\mathcal{V}(\rho)$ is trivial by \ref{propbase} (1). But this is
the case for all irreducible $\rho$
with $\dim \rho > 1$ by proposition \ref{propdimn2}. Since the case $\dim \rho = 1$ has been
dealt with separately, this concludes the proof of the theorem.

\section{Groups of small rank}

We let again $W < \GL(V)$ denote an arbitrary
irreducible (pseudo-)reflection group. In this section we
prove theorem \ref{theoW}. The proof is the same as for theorem \ref{theoSN}, as soon as
we know that $R(W)$ is generated by a collection of
representations $\rho$ for which $\mathcal{V}(\rho)$ is trivial.
This is the goal of the propositions proved below, which concludes
the proof of the theorem.

\subsection{Groups of rank 2}

The goal of this section is to prove theorem \ref{theoW} for the groups of
rank 2, that is $\dim V  = 2$.
We first assume that $W$ has type $G(de,e,2)$. For
the classical facts on this group used in the proof,
we refer to \cite{ARIKIKOIKE}.

\begin{lemma} If $W = G(de,e,2)$, then $R(W)$ is generated by $V$ and the 1-dimensional representations.
\end{lemma}
\begin{proof}
The irreducible representations of $W$ have dimension at most 2. The ones
of dimension 2 can be extended to $G(de,1,2)$, so we can assume
without loss of generality that $e=1$.
The group $W$ is generated by $t = \mathrm{diag}(1, \zeta)$ with $\zeta = \exp(2 \ii \pi/d)$ and $s$ the
permutation matrix $(1\ 2)$. Its two-dimensional
representations are indexed by couples $(i,j)$ with $0 \leq i < j < d$.
We extend this notation to $i,j \in \Z$ with $j \not\equiv i \mod d$ by taking representatives modulo
$d$ and letting $(i,j) = (j,i)$.
 A matrix model for the images of $t$ and $s$ in the
representation $(r,r+k)$  is
$$
t \mapsto \begin{pmatrix} \zeta^r & 0 \\ 0 & \zeta^{r+k} \end{pmatrix}\ \ 
s \mapsto \begin{pmatrix} 0 & 1 \\ 1 & 0 \end{pmatrix}
$$
In particular, $V = (0,1)$. From these explicit models it is straightforward to check that
$(0,1) \otimes (0,1)$ is the sum of $(0,2)$ and 1-dimensional representations,
and that $(0,1) \otimes (0,k) = (0,k+1) + (1,k)$.
Then we consider the 1-dimensional representation $\chi_1 : t \mapsto \zeta, s \mapsto 1$.
It is clear that $(i,j) \otimes \chi_1 = (i+1,j+1)$. Letting $Q$ denote
the subring of $R(W)$ generated by $V$ and the 1-dimensional representations,
through tensoring by $\chi_1$ is it enough to show that $(0,k) \in Q$ for all
$1 \leq k \leq d$. By definition $(0,1) \in Q$, tensoring by $(0,1)$
yields $(0,2) \in Q$, and finally $(0,1) \otimes (0,k) = (0,k+1) + \chi_1 \otimes (0,k-1)$
proves the result by induction on $k$.
\end{proof}

In the Shephard-Todd classification, exceptional irreducible
reflection groups of rank 2 are labelled $G_n$ for $4 \leq n \leq 22$.
For these exceptional groups, we checked from the character tables that $R(W)$
is generated by $V$ and the 1-dimensional representations,
using the following naive algorithm : start with $V$ and the
1-dimensional representations, decompose all possible tensor
products, isolate the ones which are the sum of one new irreducible
representation (with multiplicity 1) and older ones, add them
to the starting list, and repeat the process until all
irreducible representations have been obtained. In the sequel, we call
this algorithm the \emph{main algorithm}.

Together with the previous lemma, this proves the following.

\begin{prop} If $W$ is an irreducible rank 2 reflection group
then $R(W)$ is generated by $V$ and 1-dimensional representations.
\end{prop}

\subsection{Exceptional groups of higher rank}

Investigating the character tables of higher rang exceptional groups
by using the main algorithm yields the
following.

\begin{lemma} If $W$ has type $G_{24}$, $G_{25}, G_{26},G_{33}, G_{35}=E_6$,
then $R(W)$ is generated by the 1-dimensional
representations and the alternating powers $\Lambda^k V$
of the reflection representation.
\end{lemma}

The case of $G_{25}$ is specially interesting because, like
the rank 2 exceptional groups $G_4,G_8,G_{16}$, its braid
group is an ordinary braid group (here on 4 strands). The remaining
one sharing this property is $G_{32}$. For it, we need an additional trivialization.

\medskip

We consider first $W_0 \simeq \mathfrak{S}_5$ of type Coxeter $A_4$, and
the corresponding spaces $X_0$, $X_0/W_0$. We consider the
5-dimensional representation $\rho_0 : W_0 \to \GL(U)$ labelled by the partition $[3,2]$,
and introduce the 1-parameter familly of KZ-connections on $X\times U \to X$
defined by the integrable 1-form $\om = h \sum_s (\mathrm{Id}_U - \rho_0(s))
\om_s$, choose a basepoint $\underline{z} \in X_0$ and denote
$\rho_h : B_0 \to \GL(U)$ the corresponding representation of
the braid group on 5 strands $B_0 = \pi_1(X_0/W_0)$.
We denote $\sigma_1,\dots,
\sigma_4 \in B_0$ the braided reflections which are the usual Artin
generators of $B_0$. Recall that the $\sigma_i$ lie in the same
conjugacy class, and that transpositions act through $\rho_0$ on $U$
by eigenvalues $1,1,1,-1,-1$. For a formal value of the parameter $h$,
it follows that $\rho_h(\sigma_i)$ is diagonalizable with eigenvalues
$1,1,1,-e^{2 \ii \pi h},-e^{2 \ii \pi h}$ (see e.g. \cite{KRAMCRG}
proposition 2.3). In particular $\rho_h(\sigma_i)$
is annihilated by $(X-1)(X+e^{2 \ii \pi h})$ for all $h \in \C$,
hence is semisimple for $2h -1 \not\in 2 \Z$. Moreover its spectrum
is determined by the collection $\tr \rho_h(\sigma_i^k)$
for all $k$. Since $\tr \rho_h(\sigma_i^k) = 3 - 2 e^{2 \ii \pi h k}$
in $\C((h))$, the equality also holds for $h \in \C$.

In particular, for $h = 5/6$ we get that $\rho_h(\sigma_i)$
has eigenvalues $1,1,1,j,j$ with $j = \exp(2 \ii \pi/3)$.
The group $G_{32}$ is a quotient of $B_0$ by the relation $\sigma_1^3 =1$,
hence $\rho_{5/6}$ factorizes through $G_{32}$. Since $B_0$ is isomorphic
to the braid group $B$ of $W = G_{32}$ (see \cite{BMR}),
and $X/W$ as well as $X_0/W_0$ are $K(\pi,1)$,
it follows $X/W$ and $X_0/W_0$ are homotopically equivalent.
By theorem \ref{theoSN}, the vector bundle $\mathcal{V}(\rho_0)$
is trivial over $X_0/W_0$ hence over $X/W$. By proposition
\ref{propbase} (2) it follows that $\mathcal{V}(\rho_{5/6})$
is also trivial over $X/W$.

\medskip

In order to use the character table, we now identify $\rho_{5/6}$ among the representations of $G_{32}$,
and denote $\chi_h = \tr \rho_h$. We know $\chi_{5/6}(1) = 5$
and $\chi_{5/6}(s) = 3+2j = 2 + \sqrt{-3}$ for $s$ a distinguished reflection.
Let now $\beta = (\sigma_1 \sigma_2
\sigma_3 \sigma_4)^5 \in Z(B)$, which corresponds to the
loop $t \mapsto \underline{z} e^{2 \ii \pi t}$. We have $\rho_{h}(\beta)
= \exp(2\ii\pi h \sum_s (\Id + \rho(s)))$ where the
sum runs over all transpositions $s \in \mathfrak{S}_5$. Since the
sum of these transpositions is central in $Z( \C \mathfrak{S}_5)$,
one readily gets $\sum_s (\Id + \rho(s))) = 8 \Id$ and
$\rho_{5/6}(\beta) = j^2$. The irreducible characters of $G_{32}$
of degree at most 5 are either 5-dimensional, 1-dimensional
given by $\sigma_i \mapsto j^k$, or 4-dimensional. The value of these
4-dimensional characters on the distinguished reflections of $G_{32}$
belong to $\{ -1, -j, -j^2 \}$. The only possibility is thus
that $\rho_{5/6}$ is irreducible, and there is only one 5-dimensional
irreducible character satisfying the above conditions (labelled $\phi_{5,4}$
in the computer system GAP3/CHEVIE).

Launching the main algorithm yields

\begin{prop} For $W$ of type $G_{32}$, $R(W)$ is generated
by the 1-dimensional representations, $V$ and $\rho_{5/6}$.
\end{prop}

We now let $W = G_{23} = H_3$. The representation
$U = S^2 V - \un$ is irreducible, and $S^4 V = \un + 2 U + X$,
with $X$ an irreducible 4-dimensional representation.
Using the main algorithm we check
\begin{prop} For $W = G_{23} = H_3$, $R(W)$ is generated by the 1-dimensional
representations, $V$ and $X$.
\end{prop}

We let $W = G_{30} = H_4$. Launching the main algorithm on the 1-dimensional
representations and the $\Lambda^k V$ provides 25 irreducible
representations (among the 34 existing ones) inside the subalgebra $Q$
of $R(W)$ that they generate. We then consider the representations
$\mathcal{S}_{\la}(V)$ for $\la$ a partition of size at most 8, and consider
the submodule spanned by their characters and the ones of these
25 representations. Launching the LLL algorithm as implemented in
the computer system GAP3 on these characters
shows that this submodule actually contains all the 34 irreducible
representations of $W$. This proves the following.

\begin{prop} If $W$ has type  $G_{30} = H_4$,
then $R(W)$ is generated by the 1-dimensional
representations and the representations $\mathcal{S}_{\la}(V)$, for $|\la| \leq 8$.
\end{prop}

\subsection{Groups of Coxeter type $B$}

We investigate the types $B_n$ for $n$ small, as a good example
of complications arising already for close relatives of
the symmetric group. The rank 2 groups having been dealt with
before, we start with $n = 3$. In that case it is easily checked that $R(W)$
is generated by the 1-dimensional representations and the $\Lambda^k V$,
for instance by using the main algorithm.

This is not true anymore for $n =4$. In that case, it can been checked
that the subring generated by these representations has dimension
over $\Q$ equal to 18, to be compared to the 20 irreducible representations
of $W$. A fortiori, $R(W)$ is not generated over $\Z$ by these
representations.

We nevertheless show that $\mathcal{V}(\rho)$ is a trivial bundle
for all representations $\rho$ of $B_4$. For this, we first
launch the main algorithm on the 1-dimensional representations
and the $\Lambda^k V$. We get 10 irreducible representations
(among 20). The submodule $M$ spanned by them and the $S^k V$
for $2 \leq k \leq 8$ contains 4 additional ones, found by the LLL
algorithm. We did not succeed in improving this bound using other algebraic
relations and Schur functors. However, we can use a result from
\cite{ARRREFL} saying that the permutation representation of $W$ associated
to its action by conjugation on the reflections, denoted $R_0$ in
\cite{ARRREFL}, is connected
by a path in the representation variety of $B$ to another representation
$R_1$. It follows that $\Phi(R_0) = \Phi(R_1)$. By computing
the characters we get $R_1 \in M$, and we find that $R(W)$
is spanned by $M$ and $R_1 - R_0$, hence

\begin{prop} If $W$ is of type $B_n$, $n \leq 4$, then
$\Phi(\rho)$ lies inside the subgroup of trivial bundles
for any representation $\rho$ of $W$.
\end{prop}

\end{document}